\documentclass{article}
\usepackage{graphicx} 

\usepackage[utf8]{inputenc}
\usepackage{geometry}
\usepackage{amsmath, amssymb}
\usepackage{algorithm}[boxed,boxedruled]
\usepackage[algo2e]{algorithm2e}[linesnumbered,vlined]
\usepackage{mathrsfs}
\usepackage{amsthm}
\usepackage{float}
\usepackage{graphicx}
\usepackage{caption}
\usepackage{subcaption}
\usepackage{authblk}
\usepackage{orcidlink}
\usepackage{multicol}
\usepackage[numbers]{natbib}    

\newtheorem{theorem}{Theorem}
\newtheorem*{theorem*}{Theorem}
\newtheorem{lemma}{Lemma}

\begin{document}

\title{On the Optimality of CVOD-based Column Selection}

\author[a,\orcidlink{0000-0002-3209-4705}]{Maria Emelianenko \footnote{E-mail: memelian@gmu.edu}}
\author[a,\orcidlink{guyoldaker4@gmail.com}]{Guy B. Oldaker IV\footnote{E-mail: goldaker@gmu.edu} }

\affil[a]{Department of Mathematical Sciences, George Mason University, 4400 University Dr, Fairfax, VA 22030}

\maketitle


\begin{abstract}
While there exists a rich array of matrix column subset selection problem (CSSP) algorithms for use with interpolative and CUR-type decompositions, their use can often become prohibitive as the size of the input matrix increases.  In an effort to address these issues, the authors in \cite{emelianenko2024adaptive} developed a general framework that pairs a column-partitioning routine with a column-selection algorithm.  Two of the four algorithms presented in that work paired the Centroidal Voronoi Orthogonal Decomposition (\textsf{CVOD}) and an adaptive variant (\textsf{adaptCVOD}) with the Discrete Empirical Interpolation Method (\textsf{DEIM}) \cite{sorensen2016deim}.  In this work, we extend this framework and pair the \textsf{CVOD}-type algorithms with any CSSP algorithm that returns linearly independent columns. Our results include detailed error bounds for the solutions provided by these paired algorithms, as well as expressions that explicitly characterize how the quality of the selected column partition affects the resulting CSSP solution.
\end{abstract}

\section{Introduction}

Interpretable dimension reduction continues to be an important and active field of research.  The primary motivation stems from the fact that the popular techniques, such as principal component analysis (PCA) and methods based on the singular value decomposition (SVD), return transformed points that are linear combinations of potentially all of the singular vectors used in the projection.  Any physical meaning and/or attributes (e.g., non-negativity or sparsity) present in the original samples is lost \cite{mahoney2009cur}.  Tools like the interpolative (ID) and CUR decompositions  \cite{goreinov1997theory} \cite{dong2021simpler} address these issues by constructing matrix factorizations that utilize carefully selected rows/columns from the original data matrix.  The difficulty in forming such factorizations resides in determining which rows/columns to select, an issue referred to as the column-subset selection problem (CSSP) \cite{boutsidis2009improved}.  A diverse collection of deterministic and probabilistic algorithms exist for this task.  However, for many of these, especially those reliant on the SVD, their use becomes prohibitive as the problem size becomes large \cite{dong2021simpler}. To address this issue, the authors of \cite{emelianenko2024adaptive} developed a general framework for subdividing/distributing the CSSP task into a collection of smaller sub-tasks.  By first partitioning the columns of a matrix and then applying an existing CSSP algorithm to each piece, one is able to reduce the problem to a more manageable form that is well-suited for parallelization.  The partitioning algorithms considered therein include the Centroidal Voronoi Orthogonal Decomposition (\textsf{CVOD}) \cite{du2003centroidal} and Vector Quantized Principal Component Analysis (\textsf{VQPCA}), \cite{kambhatla1997dimension} \cite{kerschen2002non} \cite{kerschen2005distortion} as well as adaptive versions of each.  The Discrete Emprical Interpolation Method (\textsf{DEIM}) \cite{sorensen2016deim} is used to form the final CSSP solution. In the analysis presented in \cite{emelianenko2024adaptive}, it is unclear  how the quality of the resulting partition affects the resulting CSSP solution.  Moreover, the algorithms considered are all in terms of \textsf{DEIM}.  The objective of this paper is to extend the \textsf{CVOD}-type framework to be paired with any CSSP algorithm that yields linearly independent columns, and investigate the relationship between the CSSP reconstruction error and the optimality of the corresponding partitioning algorithm.  Our focus will be solely on pairing \textsf{CVOD} and the adaptive variant \textsf{adaptCVOD} developed in \cite{emelianenko2024adaptive} with other CSSP routines.  Our new frameworks will be referred to as \textsf{CVOD+CSSP} and \textsf{adaptCVOD+CSSP} respectively. The remainder of the article is organized as follows.  We begin with a review of the CSSP problem and several of the algorithms designed for its solution.  The section following covers the partitioned-based CSSP methods outlined in \cite{emelianenko2024adaptive}.  This is followed by our analysis of the partition/CSSP relationship and a conclusion.

\section{The Column-Subset Selection Problem (CSSP)}

Given $A \in \mathbb{R}^{m \times n}$ with $\mbox{rank}(A) = \rho$ and a positive integer $0<r\le \rho$, the goal of the column-subset selection problem (CSSP) \cite{boutsidis2009improved} is to form a matrix $C \in \mathbb{R}^{m \times r}$ using columns from $A$ that minimizes
$$ \|(I_m - CC^\dagger)A\|_\xi.$$
Here, $C^\dagger$ denotes the Moore-Penrose pseudoinvers of $C$ and $\xi$ is usually taken to be 2 or $F$.  The factor $C^\dagger A$ is referred to as the interpolative decomposition (ID) of $A$ with target rank $r$ \cite{dong2021simpler}.  We will refer to $\|(I_m - CC^\dagger)A\|_F$ as the ID error and CSSP error interchangeably.  

Determining solutions to the CSSP problem is a non-trivial task \cite{shitov2017column}, and several approaches exist that are devoted to its solution.  These include methods that sample columns probabilistically, for example via probabilities built using the norm of each column and leverage scores which use information from an SVD (\cite{drineas2006sampling} \cite{drineas2006subspace} \cite{drineas2008relative} \cite{wang2013improving}\cite{deshpande2006adaptive}\cite{mahoney2009cur}).  Others select columns based on information from the pivot elements that arise in classical matrix factorizations.  These include the LU factorization with partial pivoting (LUPP) \cite{trefethen2022numerical} and column-pivoted QR decomposition (CPQR) \cite{golub2013matrix}.  The \textsf{DEIM} algorithm also falls in this category, but avoids explicity performing row operations \cite{sorensen2016deim}.

\section{CVOD-based CSSP}

A number of the CSSP algorithms mentioned in the previous section become prohibitive as the problem size increases \cite{dong2021simpler},\cite{drineas2008relative}.  This is especially true of the SVD-based methods (e.g., \textsf{DEIM} and leverage scores \cite{voronin2017efficient}).  The authors in \cite{emelianenko2024adaptive} attempt to address this issue by partitioning the columns of the data matrix into Voronoi sets \cite{okabe2000spatial}.  This is followed by applying a CSSP algorithm to each piece and combining the results.  Two of the partitioning routines used in that paper include the Centroidal Voronoi Orthogonal Decomposition (\textsf{CVOD}) and a data-driven variant, \textsf{adaptCVOD}.  The next two sections are devoted reviewing \textsf{CVOD} and \textsf{adaptCVOD}.  We then present our new post-processing algorithm \textsf{PartionedCSSP}, which determines a CSSP solution given a column-partition of a matrix and a user-prescribed CSSP algorithm.


\noindent \textbf{Notation:  }Given a set $S \subset \mathbb{R}^m$ containing $n$ elements, we may also interpret $S$ as an $m \times n$ matrix with the $n$ elements set as columns.  Whether $S$ is interpreted as a set or a matrix will be clear from the context.  For $a \in \mathbb{R}$, $a>0$, we let $\lceil a \rceil$ denote the smallest integer that exceeds $a$, and $\lfloor a \rfloor$ denote the largest integer that does not exceed $a$.  The identity on $\mathbb{R}^m$ will be written as $I_m$ or $I_{m \times m}$.   We also let $\Omega_B$ denote the set of columns for a matrix $B$. Lastly, if $B_i \in \mathbb{R}^{m \times n_i}$, $i = 1,\ldots,k$ is a collection of matrices, we write $\mbox{diag}(B_i)$ to denote the block-diagonal matrix of size $km \times \sum_{i=1}^k n_i$:
$$\left ( \begin{array}{cccc}
          B_1 & & &\\
           & \ddots & &\\
           & & & B_k\\
           \end{array}\right )$$

\subsection{CVOD}

Introduced by Du et al., the Centroidal Voronoi Orthogonal Decomposition (\textsf{CVOD}) \cite{du2003centroidal} was originally conceived as a model-order reduction algorithm that combines elements of Centroidal Voronoi Tessellation  theory with the well-known Proper Orthogonal Decomposition (POD) \cite{gu1996efficient} technique \footnote{Also referred to as the Karhunen-Loeve expansion \cite{fukunaga2013introduction}}.  For a data matrix $A \in \mathbb{R}^{m \times n}$, positive integers $r,k$ and a multi-index $d = (d_1, \dots,d_k)^T \in \mathbb{N}^k$, the \textsf{CVOD} optimization problem is given by
$$\min_{\{(V_i,\Theta_i)\}_{i=1}^k}\mathcal{G}_1 \quad \mbox{such hat}$$
$$\Theta_i^2 = \Theta_i,\quad \mbox{rank}(\Theta_i) = d_i\quad i = 1,\ldots,k,$$
where the energy functional, $\mathcal{G}_1$, is defined as
$$\mathcal{G}_1 = \sum_{i=1}^k \sum_{x \in V_i}\|(I_m - \Theta_i)x\|_2^2.$$
Here, the $V_i \subset \Omega_A$ form a partition of the columns of $A$.  Solutions are found by using the generalized Lloyd method \cite{du1999centroidal} \cite{du2006convergence}, which performs alternating minimization.  Given an initial partition, $\{V_i\}_{i=1}^k$, one begins by determining the centroids, $\{\Theta_i\}_{i=1}^k$.  These are given by the matrices $U_i \in \mathbb{R}^{m \times d_i}$ that contains the top left singular vectors of $V_i$.  The next step is to update the $V_i$ (hereafter referred to as Voronoi sets) using the rule
$$x \in V_i \iff \|(I_m - U_iU_i^T)x\|_2^2 < \|(I_m - U_sU_s^T)x\|_2^2\quad i \neq s.$$

In the event ties occur, points are assigned to the Voronoi set with the smallest index.  Once these two steps are complete, the process repeats.  We stop the algorithm once the improvement in the energy functional value falls below a user-defined threshold, although other mechanisms can be used; see Algorithm \textsf{CVOD}.  When applied to the columns of a data matrix, $A$, \textsf{CVOD} will return an optimal column partitioning, $\{V_i\}_{i=1}^k$, and low-dimensional subspaces for each $V_i$.

\floatname{algorithm}{}
\renewcommand{\algorithmcfname}{}

\begin{algorithm}[H]
\renewcommand{\thealgorithm}{}      
\caption{\centering \textbf{Algorithm:} $\mathsf{CVOD}$}\label{alg:CVOD}
\SetAlgoRefName{\textsf{CVOD}}
\KwData{A matrix $A \in \mathbb{R}^{m \times n}$, with rank($A$) = $\rho$, a positive integer $r<\rho$, a positive integer $0<k\le m$, a multi-index of dimensions, $d = \{d_i\}_{i=1}^k$ with $\sum_{i=1}^k d_i = r$, and a positive tolerance parameter, $\epsilon$}
\KwResult{A collection, $\{V_i,U_i\}_{i=1}^{k}$, consisting of a column partitioning of $A$, and a set of lower dimensional representations of each partition. }


\vspace{0.5\baselineskip}

$\{V_i\}_{i=1}^k \leftarrow $ Randomly partition the columns of $A$\\
$j \leftarrow 1$\\
$\Delta^{j-1} \leftarrow \epsilon + 1$

\While{$\Delta^{j-1} > \epsilon$}{
    $(\{U_i\}_{i=1}^k,k) \leftarrow \mathsf{UpdateCentroidsFixed} \left (\{V_i\}_{i=1}^k, d \right )$\\    
    $ \{V_i\}_{i=1}^k \leftarrow \mathsf{FindVoronoiSets} \left (\{V_i\}_{i=1}^k,\{U_i\}_{i=1}^k\right )$
    
    $\mathcal{G}^{j} \leftarrow \sum_{i=1}^k\sum_{x \in V_i}\|(I_m-U_iU_i^T)x\|_2^2$\\
    \If{$j<2$}{
        $\Delta^{j} \leftarrow \Delta^{j-1}$
    }\Else{
        $\Delta^{j} \leftarrow \mathcal{G}^{j-1} - \mathcal{G}^{j}$
    }
    $j \leftarrow j + 1$
}
return $\{V_i,U_i\}_{i=1}^k$
\end{algorithm}

\subsection{Adaptive CVOD}

The \textsf{adaptCVOD} algorithm presented in \cite{emelianenko2024adaptive} is a data-driven version of \textsf{CVOD}.  As we show below, the difference between the two is in how they perform centroid updates.  Using the same inputs as \textsf{CVOD} except without the multi-index, the \textsf{adaptCVOD} optimization problem is given by
$$\min_{\{(V_i,\Theta_i)\}_{i=1}^k}\mathcal{G}_1\quad \mbox{such that}$$
$$\Theta_i^2 = \Theta_i,\quad \sum_{i=1}^k\mbox{rank}(\Theta_i) = r,\quad \bigcup_{i=1}^kV_i = \Omega_A.$$
The energy functional is the same for both algorithms.  However, the constraint, $\sum_{i=1}^k\mbox{rank}(\Theta_i) = r$, differs from that in \textsf{CVOD} and reflects the global approach to minimizing $\mathcal{G}_1$.  Recall that \textsf{CVOD} performs a series of local minimizations
$$\min_{\theta_i \in \mathbb{R}^{m \times d_i}}\|(I_m - \Theta_i)V_i\|_F,\quad \Theta_i^2 = \Theta_i,\quad i = 1,\ldots,k.$$
The adaptive variant instead solves
\begin{eqnarray*}
    \min_{\Phi}\|\mathsf{diag}(V_i) - \Phi \mathsf{diag}(V_i)\|^2_F&s.t.& \Phi^2 = \Phi \in \mathbb{R}^{km \times km}\\
    & & \mbox{rank}(\Phi) = r,\\
\end{eqnarray*}
which is optimal \cite{li2020elementary}.

\floatname{algorithm}{}
\begin{algorithm}[H]
\renewcommand{\thealgorithm}{}
\caption{\centering \textbf{Subroutine: }\textsf{FindVoronoiSets}}
\label{alg:findvoronoisets}
\KwData{A data matrix $A \in \mathbb{R}^{m \times n}$, with rank($A$) = $\rho$, and a set of generalized centroids, $\{U_i\}_{i=1}^k.$}
\KwResult{$\{V_i\}_{i=1}^k$, where the $V_i$ form an updated partition of the columns of $A$.}


\vspace{0.5\baselineskip}

$\Omega \leftarrow $set of column vectors of $A$\\
$k \leftarrow $Number of centroids, $U_i$\\
$V_i \leftarrow \emptyset,\; i = 1,\ldots,k$

\For{$ x \in \Omega$}{
    \For{$i = 1,\ldots,k$}{
       $d_i \leftarrow \| x - U_iU_i^Tx \|_2^2$
    }
    Assign $x$ to $V_i$ with $d_i < d_j\; i \neq j$
}
return $\{V_i\}_{i=1}^k$
\end{algorithm}

  The solution to this problem is $\mbox{diag}(U_iU_i^T)$, where $U_i$ contains the left singular vectors of $V_i$ that make up part of the rank-$r$ SVD of $\mbox{diag}(V_i)$.  Minus this change, the minimization steps are exactly the like those for \textsf{CVOD}.

The data-driven componenet of \textsf{adaptCVOD} appears when one or more of the $V_i$ have no left singular vectors that contribute to the dominant $r$-dimensional column space of $\mbox{diag}(V_i)$.  For example, we may have a situation where the rank-$r$ left singular matrix for $\mbox{diag}(V_i)$ looks like
$$\left ( \begin{array}{cccc}
          U_1U_1^T & & &\\
           & \ddots & &\\
           & & U_{k-1}U_{k-1}^T & \\
           & & & 0\\
           \end{array}\right )$$
           
Should this happen, the number of Voronoi sets, $k$, is reduced to match the number of centroids involved in this subspace.  This change allows the dimension and number of Voronoi sets to adjust to the data.    The \textsf{CVOD} pseudocode can be modified to suit \textsf{adaptCVOD} by replacing the \textsf{UpdateCentroidsFixed} routine with \textsf{UpdateCentroidsAdapt}.

 \floatname{algorithm}{}
\begin{algorithm}[H]
\label{updatecentroidsfixed}
\renewcommand{\thealgorithm}{}
\caption{\centering \textbf{Subroutine: }\textsf{UpdateCentroidsFixed}}

\KwData{A column partition, $\{Y_i\}_{i=1}^k$ of a matrix $A \in \mathbb{R}^{m \times n}$ with rank($A$) = $\rho$ and a multi-index $d = (d_1 \dots d_k)$.}
\KwResult{$\{U_i\}_{i=1}^k$, $k$, where the $U_i$ form an updated set of $k$ generalized centroids and $k$ is the number of Voronoi sets.}



\For{$i = 1,\ldots,k$}{
    $\widetilde{U}\Sigma W^T \leftarrow \mathsf{SVD}(Y_i)$\\
    $U_i \leftarrow \widetilde{U}(:,1:d_i)$
}

return $(\{U_i\}_{i=1}^k,k)$
\end{algorithm}
%

\subsection{Partitioned CSSP with Adaptive Column Selection}

The \textsf{CVOD} and \textsf{adaptCVOD} algorithms each return a partitioning, $\{V_i\}_{i=1}^k$ of the columns of the data matrix $A$ and $\{U_i\}_{i=1}^k$, which represent low-dimensional subspaces for each $V_i$. 
In \cite{emelianenko2024adaptive}, the \textsf{PartionedDEIM} algorithm applies \textsf{DEIM} to each $V_i$ and returns a combined result.  The \textsf{PartionedCSSP} algorithm presented here extends this last algorithm by allowing one to use any CSSP algorithm (including \textsf{DEIM}) that returns linearly independent columns.  These new, combined algorithms will be referred to as \textsf{CVOD+CSSP} and \textsf{adaptCVOD+CSSP}.

We represent the selected CSSP algorithm as a mapping
$$\mathcal{M}_{\textsf{CSSP}}: \mathbb{R}^{m \times n} \times \{1,\ldots,n\} \rightarrow \{1,\ldots ,n\}.$$
For example, if $A \in \mathbb{R}^{m \times n}$ and $0<r<\rho = \mbox{rank}(A)$, then $\mathcal{M}_{\textsf{CSSP}}(A,r) = \mathcal{J} \subset \{1,\ldots,n\}$ where $\mathcal{J}$ has cardinality $r$ and contains the selected column indices of $A$.  \textsf{PartionedCSSP} processes the Voronoi sets, $V_i$, in a sequential fashion in order to ensure that the returned matrix $C \in \mathbb{R}^{m \times r}$ has full column rank.  First, we sort the $V_i$ in ascending order by the ranks of their centroids; i.e, 
$$\{V_1,\ldots,V_k\} \iff \mbox{rank}(U_i) \le \mbox{rank}(U_{i+1}).$$

\begin{algorithm}[H]
\label{alg:updatecentroidsadapt}
\renewcommand{\thealgorithm}{}
\caption{\centering \textbf{Subroutine: }$\mathsf{UpdateCentroidsAdapt}$}
\KwData{A column partition, $\{Y_i\}_{i=1}^k$ of a matrix $A \in \mathbb{R}^{m \times n}$, with rank($A$) = $\rho$, and a positive integer $r \le \rho$.}
\KwResult{$(\{U_i\}_{i=1}^{\tilde{k}},\tilde{k})$, where the $U_i$ form an updated set of $\tilde{k}$ generalized centroids}


\vspace{0.5\baselineskip}

\For{$i = 1,\ldots,k$}{
    $U_i^{(0)}\Sigma_{i}^{(0)}W_i^{(0)} \leftarrow \mathsf{SVD}(Y_i)$\\
    $S_i \leftarrow \mbox{ singular values of }\Sigma_i$\\
    $U_i^{(1)} \leftarrow \emptyset$
}
$S \leftarrow \mbox{Top r singular values of } diag(\Sigma_i)$

$\widetilde{k} \leftarrow 0\\$
\For{$\sigma \in S$}{
    \For{$i = 1,\ldots,k$}{
        \If{$\sigma \in S_i$}{
            $U_i^{(1)} \leftarrow $Append corresponding column from $U_i^{(0)}$\\
            $\widetilde{k} \leftarrow \widetilde{k} + 1$
        }
    }
}

return $ \left ( \{U_i^{(1)}\}_{i=1}^{\widetilde{k}},\widetilde{k} \right)$
\end{algorithm}

\noindent Next, we run $\mathcal{M}_{\textsf{CSSP}}$ on $V_1$ to select $d_1 = \mbox{rank}(U_1)$ columns from $v_1$:
$$\mathcal{J}_1 = \mathcal{M}_{\textsf{CSSP}}(V_1,d_1),\quad C = V(:,J_1).$$
To select columns from $V_2$, we first project onto the nullspace of $C$
$$\mathcal{J}_2 = \mathcal{M}_{\textsf{CSSP}}(I_m - CC^\dagger)V_2,d_2),\quad C_2 = V_2(:,\mathcal{J}_2),$$ 
and select $d_2 = \mbox{rank}(U_2)$ columns from $V_2$.  The resulting columns are appended to the matrix $C$
and the process repeats until $C$ has $r$ columns.  As shown later, the final matrix $C \in \mathbb{R}^{m \times r}$ will have full column rank.

\floatname{algorithm}{}
\renewcommand{\algorithmcfname}{}
\begin{algorithm}[H]
\renewcommand{\thealgorithm}{}
\caption{\centering \textbf{Algorithm:} $\mathsf{PartitionedCSSP}$}
\label{alg:pCSSP}
\KwData{A column partition, $\{V_i\}_{i=1}^k$ of a matrix $A \in \mathbb{R}^{m \times n}$, with rank($A$) = $\rho$, a positive integer $r<\rho$, a collection, $\{U_i\}_{i=1}^k,$ of $m \times d_i$ matrices containing the top $d_i$ left singular vectors of each $V_i$ with $\sum_{i=1}^k d_i = r$, and a \textsf{CSSP} algorithm, $\mathcal{M}_{\textsf{CSSP}}$.}
\KwResult{$C \in \mathbb{R}^{m \times \tilde{r}}$, $\tilde{r} \le r$, such that $\|A - CC^\dagger A\|_F$ is small.}


\vspace{0.5\baselineskip}

$\{V_i\}_{i=1}^k \leftarrow $ Sort $V_i$ by $\mbox{rank}(U_i) \le \mbox{rank}(U_{i+1})$\\
$C_1 \leftarrow \mathcal{M}_{\textsf{CSSP}}(V_1,d_1)$\\
$C \leftarrow C_1$

\For{$i = 2,\ldots,k$}{
\hspace{0.4\baselineskip}$Q_{i-1}R_{i-1} \leftarrow \mbox{qr}(C)$ \tcp*{QR-decomposition}
    \hspace{0.4\baselineskip}$\tilde{V}_i\leftarrow (I_m - Q_{i-1}Q_{i-1}^T)V_i$\\
    $\mathcal{J}_i \leftarrow \mathcal{M}_{\textsf{CSSP}}(\tilde{V}_i,d_i)$\\
    $C_i \leftarrow V_i(:,\mathcal{J}_i)$\\
    $C \leftarrow [C_1 \dots C_i]$
}
return $C$
\end{algorithm}

\section{Analysis}
Our goal in this section is to construct an explicit relationship between the partitioned-based CSSP solution and the corresponding partition.  To clarify the problem, we first present the lemma and theorem from \cite{emelianenko2024adaptive} that characterize the column ID and CUR reconstruction errors resulting from the \textsf{CVOD+DEIM}/\textsf{adaptCVOD+DEIM} algorithms.  The proofs can be found in \cite{emelianenko2024adaptive}.

\begin{lemma}\label{lemma2}
Let $A \in \mathbb{R}^{m \times n}$ with $\mbox{rank}(A) = \rho$, and let $0 < r < \rho$ be a desired target rank.  Let $C \in \mathbb{R}^{m \times r}$ be the matrix resulting from any of the partition-based \textsf{DEIM} algorithms with an initial column partition of size $k$ and multi-index $d = (d_1 \dots d_k)$, with $d_i = \lfloor r/k \rfloor$.  
If $\{V_i\}_{i=1}^{\tilde{k}}$ is the final column partition with $\tilde{k}\le k$, then
$$\|(I_m - CC^\dagger)A\|_F\le \sqrt{\tilde{k}\gamma_C}\|A - A_r\|_F,$$
    where $\gamma_C = \max_i\|(I_m - C_iC_i^\dagger)V_i\|_F^2\sigma_\rho^{-2}$ and $\sigma_1 \ge \sigma_2 \ge \ldots \ge \sigma_{\rho}>0$ are the singular values of $A$, $C_i \in \mathbb{R}^{m \times \tilde{d}_i}$ contains the columns of $C$ selected from $V_i$, and $A_r \in \mathbb{R}^{m \times n}$ denotes the best rank $r$ approximation to $A$ given by the truncated SVD.
\end{lemma}

\begin{theorem}
 Let $A \in \mathbb{R}^{m \times n}$ with $\mbox{rank}(A) = \rho$, and let $0 < r < \rho$ be a desired target rank.  Suppose $C \in \mathbb{R}^{m \times r}$ and $R \in \mathbb{R}^{r \times n}$ are the result from applying any of the partition-based \textsf{DEIM} algorithms on $A$ and $A^T$ respectively, each with an initial partition of size $k$ and multi-index defined as in Lemma \ref{lemma2}.  If $\{V_i\}_{i=1}^{k_1}$ and $\{W_j\}_{j=1}^{k_2}$ denote the respective final column and row partitions with $k_1,k_2 \le k$, then
 $$\|A - CUR\|_F \le\left ( \sqrt{k_1\gamma_C} + \sqrt{k_2\gamma_R} \right ) \|A - A_r\|_F,$$
 where
 $$\gamma_C = \max_i \|(I_m - C_iC_i^\dagger)V_i\|_F^2\sigma_\rho^{-2},\quad \gamma_R = \max_i \|W_j(I_n - R_j^\dagger R_j)\|_F^2\sigma_\rho^{-2}$$
 are from Lemma \ref{lemma2} and $A_r \in \mathbb{R}^{m \times n}$ denotes the best rank $r$ approximation to $A$ given by the truncated SVD.
\end{theorem}

The main issue here is that the results are, with the exception of the $\gamma_C$ and $\gamma_R$ terms, partition-agnostic.  In other words, the results are valid given any partitioning of the columns of $A$.  What we require is a result inherently tied to the choice  of partitioning algorithm.  This will be the focus of our work below.  We begin by presenting several results that will help with our proofs later on.  The next goal will be to place the column ID reconstruction error in terms of the energy functional from the corresponding partitioning strategy.  Following this, we will bound the energy functional value at termination by objects related to the data matrix under discussion.  This last will allow us to combine the results and form a more-informative bound on the column ID reconstruction error.  

In what follows, the matrix under discussion will be $A \in \mathbb{R}^{m \times n}$ with $\mbox{rank}(A) = \rho$ and target rank $0<r<\rho$.  The number of Voronoi sets will be denoted by $k$.

\subsection{Preliminaries}

In this section we cover helpful lemmas etc. that will be used for the detailed analysis that follows.  We begin with a modification of a subspace distance theorem from \cite{golub2013matrix}.  This result will allow us to relate the local reconstruction errors of each point due to the \textsf{CVOD}/\textsf{adaptCVOD} routines to the best $r$-dimensional reconstruction error of the data matrix $A$.

\begin{theorem}
\label{golub}
    Suppose $$W = [\underbrace{W_1}_{k}\: | \: \underbrace{W_2}_{n-k}], \; Z = [\underbrace{Z_1}_{k}\: | \: \underbrace{Z_2}_{n-k} ],$$
    are $n\times n$ orthogonal matrices. Then
    $$\|W_1W_1^T - Z_1Z_1^T\|_F = \|W_1^TZ_2\|_F = \|Z_1^TW_2\|_F.$$
\end{theorem}

\begin{proof}
    Following the approach from \cite{golub2013matrix}, observe that
    \begin{eqnarray*}
        \|W_1W_1^T - Z_1Z_1^T\|_F^2&=&\|W^T(W_1W_1^T - Z_1Z_1^T)Z\|_F^2\\
        &=& \left \| \left [ \begin{array}{cc}
                         0&W_1^TZ_2\\
                         -W_2^TZ_1&0\\
                         \end{array} \right ] \right \|_F^2\\
    \end{eqnarray*}
    Now note that the matrices $W_2^TZ_1$ and $W_1^TZ_2$ are submatrices of the $n \times n$ orthogonal matrix
    $$Q = \left [ \begin{array}{cc}
                   Q_{11}&Q_{12}\\
                   Q_{21}&Q_{22}\\
                   \end{array} \right] = \left [ \begin{array}{cc}
                   W_1^TZ_1&W_1^TZ_2\\
                   W_2^TZ_1&W_2^TZ_2\\
                   \end{array} \right] = W^TZ.$$
    We need to show that $||Q_{21}||_F = ||Q_{12}||_F.$ Since $Q$ has orthogonal columns, we have
    $$\left \| \left [ \begin{array}{c}
                      Q_{11}\\
                      Q_{21}\\
                      \end{array} \right ] \right \|_F^2 = k = \|Q_{11}\|_F^2 + \|Q_{21}\|_F^2 \Rightarrow \|Q_{21}\|_F^2 = k - \|Q_{11}\|_F^2.$$
    Similarly, using $Q^T,$ which is also orthogonal, we have
    \begin{eqnarray*}
        \left \| \left [ \begin{array}{c}
                      Q^T_{11}\\
                      Q^T_{12}\\
                      \end{array} \right ] \right \|_F^2 = k &=& \|Q^T_{11}\|_F^2 + \|Q^T_{12}\|_F^2 \\
                      &=&\|Q_{11}\|_F^2 + \|Q_{12}\|_F^2\\
                      &\Rightarrow&\|Q_{12}\|_F^2 = k - \|Q_{11}\|_F^2\\
    \end{eqnarray*}
    
    Thus $\|W_2^TZ_1\|_F = \|W_1^TZ_2\|_F$ and the proof is complete.
\end{proof}

Our next result bounds the discrepancy between the dominant $r$-dimensional column space of a matrix and that of a linearly independent subset of columns (from the same matrix) of size $r$.

\begin{lemma}\label{bound}
    Let $A \in \mathbb{R}^{m \times n}$ have rank $\rho$, and let $A_r = U_r \Sigma_r W_r^T$, $r < \rho$, be its truncated SVD.  If $C \in \mathbb{R}^{m \times r}$ is built using columns from $A$ and has full column rank, then
    $$\|U_rU_r^T - CC^\dagger \|_F \le \|A - A_r\|_F\|C^\dagger\|_2.$$
\end{lemma}
\begin{proof}
    Since $C$ has full column rank, we may write $C = QR$, where $Q \in \mathbb{R}^{m \times r}$ has orthonormal columns and $R \in \mathbb{R}^{r \times r}$ is upper triangular and nonsingular.  Then, $CC^\dagger = QQ^T$.  Let $\bar{U}_r\in \mathbb{R}^{m \times n-r}$ have as columns the left singular vectors of $A$ not contained in $U_r$.   These objects and theorem \ref{golub} imply
    \begin{eqnarray*}
        \|U_rU_r^T -CC^\dagger\|_F &=&\|U_rU_r^T - QQ^T\|_F\\
        &=& \|\bar{U}_r^TQ\|_F\\
        &=&\|\bar{U}_rCR^{-1}\|_F\\
        &\le&\|\bar{U}_rC\|_F \|R^{-1}\|_2\\
        &=&\|\bar{U}^T U\Sigma W^T S_c\|_F \|R^{-1}\|_2\\
    \end{eqnarray*}
    \noindent where $U\Sigma W^T = A$ is the SVD of $A$ and $S_c \in \mathbb{R}^{n \times r}$ is the column selection matrix for $C$; i.e., $C = AS_c$.  We have
    $$\bar{U}^T U\Sigma W^T = \bar{\Sigma}\bar{W}^T$$
    where $\bar{\Sigma} \in \mathbb{R}^{\rho - r \times \rho-r}$ contains the $\rho - r$ smallest singular values of $A$ and $\bar{W} \in \mathbb{R}^{n \times \rho - r}$ contains the corresponding right singular vectors.  Thus, $\|\bar{U}^T U\Sigma W^T S_c\|_F \le \|A - A_r\|_F$.  The lemma follows by recognizing that $C^\dagger = R^{-1}Q^T$.
\end{proof}

For our last result in this section, we show that the output from the \textsf{PartionedCSSP} algorithm has full column rank.

\begin{lemma}\label{independent}
    Let $A \in \mathbb{R}^{m \times n}$ with $\mbox{rank}(A) = \rho$, and let $\mathcal{M}_{\textsf{CSSP}}$ be any CSSP algorithm that returns linearly independent columns.  Let $C \in \mathbb{R}^{m \times r}$ be the result from applying \textsf{PartionedCSSP} with $\mathcal{M}_{\textsf{CSSP}}$ on the size $k$ column partition $\{V_i\}_{i=1}^k$ of $A$ with target rank $0<r<\rho$, and a multi-index $d = (d_1 \dots d_k)$.  Then $C$ has full column rank.
\end{lemma}

\begin{proof}
    Let $\{V_i\}_{i=1}^{\tilde{k}}$ denote the column partition that results when the partition algorithm completes, and assume they have been ordered as in \textsf{PartitionCSSP}.  We may write $C = [C_1 \dots C_{\tilde{k}}],$ where $C_i \in \mathbb{R}^{m \times d_i}$ contains those columns of $C$ that belong to $V_i$.  Since $C_1$ results from applying $\mathcal{M}_{\textsf{CSSP}}$ to $V_1$,  we know that it has full column rank.  Proceeding by induction, suppose $C = [C_1 \dots C_s]$, $1<s<\tilde{k}$ has been constructed and has full column rank.  We next consider $B = (I_m - QQ^T)V_{s+1} \in \mathbb{R}^{m \times d_{s+1}}$, where $QR = C$ is the QR-decomposition of $C$.  Let $T_{s+1} = I_{n \times n}(:,\mathcal{J}_{s+1}) \in \mathbb{R}^{n_{s+1} \times d_{s+1}}$, where $\mathcal{J}_{s+1} = \mathcal{M}_{\textsf{CSSP}}(B,d_{s+1})$.  Then $BT_{s+1}$ has full column rank, and each column is linearly independent with respect to the columns of $C$.  Now suppose that $V_{s+1}T_{s+1}$ does not have full column rank.  Then there exists $x \neq 0$ in $\mathbb{R}^{d_{s+1}}$ such that $V_{s+1}T_{s+1}x = 0$.  But this implies
    $$\|Bx\|_2 = \|(I_m - QQ^T)V_{s+1}x\|_2 \le \|V_{s+1}T_{s+1}x\|_2 = 0,$$
    a contradiction.  Thus, the $V_{s+1}T_{s+1}$ has full column rank.  
\end{proof}

\subsection{Error Bounds}

The first goal of this section is to bound the ID error in terms of the energy functional value that \textsf{CVOD} or \textsf{adaptCVOD} achieves when run to completion.  The next goal will be to construct an upper bound on the \textsf{CVOD} (\textsf{adpatCVOD}) energy at termination in terms of objects related to the input data matrix, $A$.  Once complete, these two results will be combined to give an overall bound on the ID reconstruction error.

Our first theorem states that the ID reconstruction error that results from either \textsf{CVOD+CSSP} or \textsf{adaptCVOD+CSSP} is on the order of the \textsf{CVOD}/\textsf{adaptCVOD} energy functional value at termination.

\begin{theorem}\label{csspError}
    Let $A \in \mathbb{R}^{m \times n}$, rank($A$) = $\rho$, and $0<r<\rho$, $0<k<n$ be integers.  If $C \in \mathbb{R}^{m \times r}$ is the output from \textsf{CVOD+CSSP} (\textsf{adaptCVOD+CSSP}), then
    $$\|(I_m - CC^\dagger)A\|_F \sim \mathcal{O}  \left (\mathcal{G}^* \right )$$
    where $\mathcal{G}^*$ is the energy value of \textsf{CVOD} (\textsf{adaptCVOD}) at completion.
\end{theorem}

\begin{proof}
For each $i = 1,\ldots,k$, let $C_i \in \mathbb{R}^{m \times d_i}$ be the submatrix of $C$ whose columns belong to $V_i$, and let $Q_iR_i = C_i$ be its QR-decomposition.  Define $\hat{U}_i \in \mathbb{R}^{m \times d_i}$, $i = 1,\ldots,k$, to be the matrix whose columns contain the top $d_i$ left singular vectors of $V_i$ i.e., the centroid of $V_i$.  Then
    \begin{eqnarray*}
        \|(I_m - CC^\dagger)A\|_F^2 &=& \sum_{i=1}^k\|(I_m - CC^\dagger)V_i\|_F^2\\
        &\le & \sum_{i=1}^k \|(I_m - C_iC_i^\dagger)V_i\|_F^2\\
        &=& \sum_{i=1}^k \|(I_m - Q_iQ_i^T)V_i\|_F^2\\
        &=&\sum_{i=1}^k \|(I_m - \hat{U}_i\hat{U}_i^T +\hat{U}_i\hat{U}_i^T - Q_iQ_i^T)V_i\|_F^2\\
        &\le&\sum_{i=1}^k \left ( \|(I_m - \hat{U}_i\hat{U}_i^T)V_i\|_F + \|(\hat{U}_i\hat{U}_i^T - Q_iQ_i^T)V_i\|_F \right )^2\\
        &\le& \sum_{i=1}^k\left ( \|(I_m - \hat{U}_i\hat{U}_i^T)V_i\|_F + \|(\hat{U}_i\hat{U}_i^T - Q_iQ_i^T)\|_F\|V_i\|_2 \right )^2\\
        &\le & \sum_{i=1}^k\left ( \|(I_m - \hat{U}_i\hat{U}_i^T)V_i\|_F + \|C_i^\dagger\|_2\|V_i - V_{i,d_i}\|_F\|V_i\|_2 \right )^2\\
    \end{eqnarray*}

The last line follows from invoking lemma \ref{bound}, where $V_{i,d_i}\in \mathbb{R}^{m \times n_i}$ denotes the best rank $d_i$ approximation to $V_i$ given by its truncated SVD.  Since $\|(I - \hat{U}_i\hat{U}_i^T)V_i\|_F = \|V_i - V_{i,d_i}\|_F$, this last gives
\begin{eqnarray*}
    \|(I_m - CC^\dagger)A\|_F^2 &\le& \sum_{i=1}^k \|(I_m - U_iU_i^T)V_i\|^2_F(1 + \|C_i^\dagger\|_2\|V_i\|_2)^2\\
    &\le& \zeta^2 \sum_{i=1}^k \|(I_m - U_iU_i^T)V_i\|^2_F\\
\end{eqnarray*}
where $\zeta \equiv \max_i \left (1 + \|C_i^\dagger\|_2\|V_i\|_2 \right ).$
Since this last result bounds the reconstruction error by a constant times the \textsf{CVOD} energy, the proof is complete.  
\end{proof}

We remark that the $\zeta$ term characterizes the local performance of the selected CSSP algorithm in terms of the conditioning of the selected columns.  This term could be used to guide the choice of CSSP algorithm to use; e.g., its form is similar to expressions found with strong rank-revealing QR-factorizations \cite{gu1996efficient}. We also note that the result is independent of the size of the column partition of $A$.

Our next theorem constructs an upper bound on the \textsf{CVOD} (\textsf{adpatCVOD}) energy at termination in terms of objects related to the input data matrix, $A$.

\begin{theorem}\label{energyError}
    Let $\{(V_i,\hat{U}_i)\}_{i=1}^k$ denote the Voronoi sets and centroids resulting from running either the \textsf{CVOD} or \textsf{adpatCVOD} algorithm on a matrix $A \in \mathbb{R}^{m \times n}$ with target rank $0<r<\mbox{rank}(A)$.  Let $d = \{d_i\}_{i=1}^k$ denote the centroid dimensions at termination.  Then
    $$\sum_{i=1}^k \sum_{x \in V_i} \|(I_m - \hat{U}_i\hat{U}_i^T)x\|_2^2 \le \|A - A_r\|_F^2 + \left (1 - \frac{1}{L^*}\right )\|A_r\|_F^2   $$
    where $L^* = \sup \{ \lceil \frac{r}{d_i} \rceil \; | \; i = 1,\ldots,k \}$ and $A_r \in \mathbb{R}^{m \times n}$ denotes the best rank-$r$ approximation to $A$ given by the truncated SVD.
\end{theorem}
\begin{proof}
    Let $\mathcal{U}_r \in \mathbb{R}^{m \times r}$ be the matrix containing the top $r$ left singular vectors of the matrix $A$.  Since each $\hat{U}_i$ has rank $d_i$, we may partition the columns of $\mathcal{U}_r$ as
$$\mathcal{U}_r = [\mathcal{U}_{r,i_1} \cdots \mathcal{U}_{r,i_{L_i}}],$$
where $\mbox{rank}(\mathcal{U}_{r,i_s})\le d_i$ for each $i_l$, $l = 1,\ldots, L_i$, where $L_i = \lceil \frac{r}{d_i} \rceil.$
Since $\hat{U}_i \neq \mathcal{U}_{r,i_l}$ for every $l$, we have
$$\|(I_m - \hat{U}_i\hat{U}_i^T)V_i\|_F^2 \le \|(I - \mathcal{U}_{r,i_l}\mathcal{U}_{r,i_l}^T)V_i\|_F^2,\quad \mbox{for every } l.$$
And since $U_{r,i_l}^TU_{r,i_l} = I$, for each $l$, this implies that
$$\|\hat{U}_i\hat{U}_i^TV_i\|_F^2 \ge \|\mathcal{U}_{r,i_l}\mathcal{U}_{r,i_l}^TV_i\|_F^2 \mbox{ for each }l.$$
Thus, 
$$\sum_{l=1}^{L_i}\|\hat{U}_i\hat{U}_i^TV_i\|_F^2 = L_i\|\hat{U}_i\hat{U}_i^TV_i\|_F^2 \ge \sum_{l=1}^{L_i}\|\mathcal{U}_{r,i_l}\mathcal{U}_{r,i_l}^TV_i\|_F^2 = \|\mathcal{U}_r\mathcal{U}_r^TV_i\|_F^2.$$

Note that we can repeat this construction for each $V_i$, $i = 1,\ldots, k$.  Let $L^* = \sup \{L_i \; | \; i = 1,\ldots,k\}$.  Then the previous shows $L^*\|\hat{U}_i\hat{U}_i^TV_i\|_F^2 \ge \|U_rU_r^TV_i\|\quad \forall i.$  As a result, we have
\begin{eqnarray*}
    \sum_{i=1}^k\|(I_m - \hat{U}_i\hat{U}_i^T)V_i\|_F^2 &=& \sum_{i=1}^k \left ( \|V_i\|_F^2 - \|\hat{U}_i\hat{U}_i^TV_i\|_F^2 \right )\\
    &=& \|A\|_F^2 - \sum_{i=1}^k\|\hat{U}_i\hat{U}_i^TV_i\|_F^2\\
    &\le&\|A\|_F^2 - \frac{1}{L^*}\sum_{i=1}^k\|\mathcal{U}_r\mathcal{U}_r^TV_i\|_F^2\\
    &=& \|A\|_F^2 - \frac{1}{L^*}\|\mathcal{U}_r\mathcal{U}_r^TA\|_F^2\\
    &=& \|A\|_F^2 - \frac{1}{L^*}\|A_r\|_F^2\\
    &=& \|A\|_F^2 - \|A_r\|_F^2 - \frac{1}{L^*}\|A_r\|_F^2 + \|A_r\|_F^2\\
    &=& \|A - A_r\|_F^2 + \left (1 - \frac{1}{L^*}\right )\|A_r\|_F^2 \\
\end{eqnarray*}

\end{proof}

By combining Theorem \ref{csspError} and \ref{energyError}, we arrive at the following ID reconstruction error bounds.

\begin{theorem}
    Let $A \in \mathbb{R}^{m \times n}$, rank($A$) = $\rho$, and $0<r<\rho$, $0<k<n$ be integers.  Define $L^*$ as in Theorem \ref{energyError} and $\zeta$ as the proof of Theorem \ref{csspError}.  If $C \in \mathbb{R}^{m \times r}$ is the output from \textsf{CVOD+CSSP} (\textsf{adaptCVOD+CSSP}), then
    $$\|(I_m - CC^\dagger)A\|_F \le \zeta \left (\|A - A_r\|_F^2 + \left (1 - \frac{1}{L^*}\right )\|A_r\|_F^2   \right )^{1/2}.$$
\end{theorem}

\noindent \textbf{Remark:  } Although this result has not been optimized, it still presents an interesting bound.  In particular, it consists of two terms that bring together elements from the CSSP algorithm and the data matrix, $A$.  The $\zeta$ term, as mentioned earlier, quantifies the local performance of the chosen CSSP algorithm in terms of the conditioning of the columns selected from each $V_i$.  The remaining term relates the ideal $r$-dimensional reconstruction error of $A$ to the partitioning algorithm's energy functional value at termination.  Of note is that the bound is independent of $k$, the number of final Voronoi sets.

\section{Conclusion}
In this work, we present generalizations of the \textsf{CVOD+DEIM}/\textsf{adaptCVOD+DEIM} algorithms introduced in \cite{emelianenko2024adaptive} designed to address the column subset selection problem (CSSP).  Referred to as \textsf{CVOD+CSSP}/\textsf{adaptCVOD+CSSP}, these new frameworks pair \textsf{CVOD}/ \textsf{adaptCVOD} with any column-selection algorithm whose output gives linearly independent columns.  We establish a quantitative relationship between the final CSSP solution and the optimality of the partitioning algorithm.  Furthermore, we develop bounds on the \textsf{CVOD}/\textsf{adaptCVOD} energy functional values at termination in terms of objects from the parent data matrix.  This last may be of independent interest in the model order reduction community \cite{du2003centroidal}, \cite{burkardt2006centroidal}.  These results allow one to interpret the CSSP error in terms of the partition quality and the local performance of the chosen CSSP method.  This result reflects the belief that the ID reconstruction error resulting from a partitioned-based CSSP procedure should improve with the quality of the underlying partition.  

Topics for future work include developing analogous generalizations using the \textsf{VQPCA} and \textsf{adaptVQPCA} partitioning algorithms, as well as conducting a numerical study that investigates the performance of \textsf{CVOD+CSSP}/\textsf{adaptCVOD+CSSP} when paired with several well-known column-selection methods.

%


\bibliographystyle{achemso}
\bibliography{references}

\end{document}